\documentclass[12pt]{article}
\usepackage{amsmath, amsfonts, tikz, geometry, ytableau, mathrsfs,amssymb}
\usepackage[indent,margin=1cm]{caption}
\usepackage{fullpage}
\usepackage[colorlinks,linkcolor=blue,anchorcolor=blue,citecolor=blue]{hyperref}
\usepackage{enumitem}
\usepackage{amsthm}

\numberwithin{equation}{section}
\numberwithin{figure}{section}

\hypersetup{colorlinks = true}
\newtheorem{thm}{Theorem}[section]
\newtheorem{conj}[thm]{Conjecture}

\newtheorem{obs}[thm]{Observation}
\newtheorem{cor}[thm]{Corollary}
\newtheorem{lem}[thm]{Lemma}
\newtheorem{prop}[thm]{Proposition}

\makeatletter
\renewenvironment{proof}[1][]{\par
 \pushQED{\qed}%
 \normalfont \topsep6\p@\@plus6\p@\relax
 \trivlist
 \item\relax
 {\itshape
 #1\@addpunct{.}}\hspace\labelsep\ignorespaces
 }{%
 \popQED\endtrivlist\@endpefalse
 }
 \makeatother
\allowdisplaybreaks

\allowdisplaybreaks

\usepackage[backend=biber,maxbibnames=99,sorting=nyt,giveninits,maxnames=10,block=space]{biblatex}
\begin{document}
\begin{center}
	{\large \bf A symmetric function approach to log-concavity of independence polynomials}
\end{center}

\begin{center}
Ethan Y.H. Li$^{1}$, Grace M.X. Li$^{2}$,  Arthur L.B. Yang$^{3}$ and Zhong-Xue Zhang$^{4}$\\[6pt]
\end{center}

\begin{center}
$^{1}$School of Mathematics and Statistics,\\
Shaanxi Normal University, Xi'an, Shaanxi 710119, P. R. China

$^{2}$School of Mathematics and Data Science, \\
Shaanxi University of Science and
Technology, Xi’an, Shaanxi 710021, P. R. China

$^{3,4}$Center for Combinatorics, LPMC\\
Nankai University, Tianjin 300071, P. R. China\\[6pt]

Email: $^{1}${\tt yinhao\_li@snnu.edu.cn}, $^{2}${\tt grace\_li@sust.edu.cn}, $^3${\tt yang@nankai.edu.cn}, $^{4}${\tt zhzhx@mail.nankai.edu.cn}
\end{center}

\noindent\textbf{Abstract.}
As introduced by Gutman and Harary, the independence polynomial of a graph serves as the generating polynomial of its independent sets. In 1987, Alavi, Malde, Schwenk and Erd\H{o}s conjectured that the independence polynomials of all trees are unimodal. In this paper we come up with a new way for proving log-concavity of independence polynomials of graphs by means of their chromatic symmetric functions, which is inspired by a result of Stanley connecting properties of polynomials to positivity of symmetric functions. This method turns out to be more suitable for treating trees with irregular structures, and as a simple application we show that all spiders have log-concave independence polynomials,  which provides more evidence for the above conjecture. Moreover, we present two symmetric function analogues of a basic recurrence formula for independence polynomials, and show that all pineapple graphs also have log-concave independence polynomials.



\section{Introduction}\label{sec-intro}

Given a graph $G = (V(G),E(G))$, an independent set (or stable set) $I$ is defined as a subset of $V(G)$ such that any two vertices of $I$ are not adjacent. Let $\alpha(G)$ denote the size of maximum independent set(s) of $G$.
The \textit{independence polynomial} (also called \textit{stable set polynomial}) of $G$, introduced by Gutman and Harary \cite{GH83}, is defined as
\[
I_G(t) = i_0+i_1t+\cdots+i_{\alpha(G)}t^{\alpha(G)},
\]
where $i_j$ denotes the number of independent sets of $G$ with cardinality $j$ and by convention $i_0 = 1$.

In 1987, Alavi, Malde, Schwenk and Erd\H{o}s \cite{AMSE87} further studied these polynomials and particularly investigated their unimodality. Recall that if $I_G(t)$ is of degree $d$ and there exists an index $0 \le j \le d$ such that $i_0 \le \cdots\le i_{j-1} \le i_j \ge i_{j+1} \ge \cdots \ge i_d$, then it is said to be \textit{unimodal}. They made the following conjecture.
\begin{conj}[{\cite[Problem 3]{AMSE87}}]\label{conj-uni}
The independence polynomial of every tree is unimodal.
\end{conj}

This celebrated conjecture has received considerable attention. For more recent work on this conjecture,  see \cite{BES18,BG21,Ben18,BH24,BC18,GH18,Ham90,HLMP24,LM03,ZZ23,Zhu22,ZW20} and references therein.

In particular, it was surmised that these polynomials are actually log-concave \cite{LM04}, which is a stronger property than unimodality and requires that $i_j^2 \ge i_{j-1}i_{j+1}$ for any $1 \le j \le d-1$. However, Kadrawi, Levit, Yosef and Mizrachi \cite{KL23, KLYM23} presented many counterexamples with unimodal but not log-concave independence polynomials. The real-rootedness of independence polynomials, a stronger property than log-concavity, has also been studied extensively;
see \cite{BHN04,CS07,HL72,LTZ25,WZ11,ZZ24,Zhu07} for instance.

In this paper we proceed to investigate this conjecture by means of chromatic symmetric functions of graphs. In particular, we find a new way for proving the log-concavity of independence polynomials, which is inspired by the following result of Stanley.

Let $s_{\lambda}(\mathbf{x})$ and $e_{\lambda}(\mathbf{x})$  denote the Schur function and the elementary symmetric function in the variables $\mathbf{x} = \{x_1,x_2,\ldots\}$ respectively.

Let $P(t)$ be a polynomial with real coefficients satisfying $P(0) = 1$, say
$$P(t) =a_0 + a_1t + \cdots + a_dt^d.$$
Stanley defined an inhomogeneous symmetric function
$$F_P(\mathbf{x}) = \prod_{i\geq 1} P(x_i)$$ and established a relation between real-rootedness of $P(t)$ and positivity of $F_P(\mathbf{x})$.

\begin{thm}[{\cite[Theorem 2.11]{Sta98}}]\label{thm-spos-rr}
Let $P(t)$ and $F_P(\mathbf{x})$ be defined as above. Then the following conditions are equivalent:
\begin{itemize}
  \item The coefficient of $s_{\lambda}(\mathbf{x})$ in $F_P(\mathbf{x})$ is nonnegative for each integer partition $\lambda$.
  \item The coefficient of $e_{\lambda}(\mathbf{x})$ in $F_P(\mathbf{x})$ is nonnegative for each integer partition $\lambda$.
  \item All zeros of $P(t)$ are negative real numbers.
\end{itemize}
\end{thm}

By restricting this result to the nonnegativity of the coefficients of $s_{\lambda}$ indexed by partitions of length at most 2, which we call \textit{2-Schur-positivity}, we establish the following equivalence relation between 2-Schur-positivity of $F_P(\mathbf{x})$ and log-concavity of $P(t)$.

\begin{thm}\label{thm-2s-lc}
Let $P(t)$ and $F_P(\mathbf{x})$ be defined as above with $a_0,\ldots,a_d$ being positive. Then the following conditions are equivalent:
\begin{itemize}
  \item[\textup{(i)}] The coefficient of $s_{\lambda}$ in $F_P(\mathbf{x})$ is nonnegative for any partition $\lambda$ of length at most~2.
  \item[\textup{(ii)}] The coefficient of $s_{(k,k)}$ in $F_P(\mathbf{x})$ is nonnegative for any $k \ge 1$.
  \item[\textup{(iii)}] $P(t)$ is log-concave.
  \item[\textup{(iv)}] $P(t)$ is strongly log-concave.
\end{itemize}
\end{thm}

As shown by Stanley \cite{Sta98}, if $P(t)$ is taken to be the independence polynomial $I_G(t)$ of a graph $G$, then $F_P(\mathbf{x})$ is just the symmetric function $Y_G = Y_G(\mathbf{x}) = \sum_{\alpha}
X_G^{\alpha}$ (to be defined in Section \ref{sec-relation}). 
Hence, this result actually establishes an equivalence relation between 2-Schur-positivity of $Y_G$ and the log-concavity of $I_G(t)$.

Based on this result, we come up with a new way for proving log-concavity of $I_G(t)$: firstly we analyze the coefficients of $s_{\lambda}$'s in each $X_G^{\alpha}$ and then show that all the negative terms could be eliminated. Unexpectedly, this method is more suitable for treating trees with irregular structures, and as an application we are able to prove that all spiders
have log-concave independence polynomials, which provides further evidence for Conjecture~\ref{conj-uni}. Precisely, a spider is a tree with only one vertex of degree at least 3 (called the \textit{torso}), or equivalently, a vertex with several paths (called \textit{legs}) pendent to it, see the graph on the left of Figure \ref{fig-sp-pi}. We should remark that in 2004 Levit and Mandrescu \cite{LM04} proved that well-covered spiders have log-concave independence polynomials, where a graph is called well-covered if all of its maximal independent sets are also maximum.

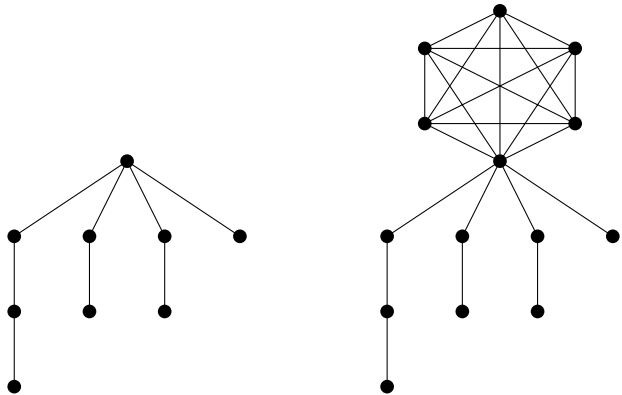
\begin{figure}[htbp]
\centering
\begin{tikzpicture}[scale = 1]

\fill (-6,0) circle (0.5ex);
\fill (-3,0) circle (0.5ex);
\fill (-4,0) circle (0.5ex);
\fill (-4,-1) circle (0.5ex);

\fill (-4.5,1) circle (0.5ex);
\fill (-5,0) circle (0.5ex);
\fill (-6,-1) circle (0.5ex);
\fill (-6,-2) circle (0.5ex);
\fill (-5,-1) circle (0.5ex);
\draw (-4.5,1) -- (-6,0) -- (-6,-1) -- (-6,-2);
\draw (-4.5,1) -- (-5,0) -- (-5,-1);
\draw (-4.5,1) -- (-4,0) -- (-4,-1);
\draw (-4.5,1) -- (-3,0);
\end{tikzpicture}
\hspace{1.5cm}
\begin{tikzpicture}[scale = 1]

\fill (-6,0) circle (0.5ex);
\fill (-3,0) circle (0.5ex);
\fill (-4,0) circle (0.5ex);
\fill (-4,-1) circle (0.5ex);

\fill (-4.5,1) circle (0.5ex);
\fill (-5,0) circle (0.5ex);
\fill (-6,-1) circle (0.5ex);
\fill (-6,-2) circle (0.5ex);
\fill (-5,-1) circle (0.5ex);
\draw (-4.5,1) -- (-6,0) -- (-6,-1) -- (-6,-2);
\draw (-4.5,1) -- (-5,0) -- (-5,-1);
\draw (-4.5,1) -- (-4,0) -- (-4,-1);
\draw (-4.5,1) -- (-3,0);

\fill (-5.5,1.5) circle (0.5ex);
\fill (-3.5,1.5) circle (0.5ex);
\fill (-5.5,2.5) circle (0.5ex);
\fill (-3.5,2.5) circle (0.5ex);
\draw (-4.5,3) -- (-5.5,2.5)  --(-5.5,1.5) -- (-4.5,1) -- (-3.5,1.5) -- (-3.5,2.5) -- (-4.5,3);
\draw (-5.5,2.5) -- (-3.5,2.5) -- (-5.5,1.5) --(-3.5,1.5) --(-5.5,2.5) -- (-4.5,1) -- (-3.5,2.5) ;
\fill (-4.5,3) circle (0.5ex);
\draw (-4.5,1) -- (-4.5,3);
\draw (-3.5,1.5) -- (-4.5,3) -- (-5.5,1.5);
\end{tikzpicture}
\caption{Spider $S(3,2,2,1)$ and pineapple graph $Pi(6,(3,2,2,1))$
}\label{fig-sp-pi}
\end{figure}



Further, we give two symmetric function analogues of the following well-known recurrence formula:
\begin{align}\label{eq-rec}
I_G(t) = I_{G-v}(t) + tI_{G-N[v]}(t),
\end{align}
where $G-v$ is the graph obtained by deleting $v$ from $G$ and $N[v] = \{v\} \cup N(v)$ consists of $v$ and all its neighbors. As an application, we prove that each pineapple graph also has a log-concave independence polynomial, which is obtained by joining a complete graph to the torso of a spider, see the graph on the right of Figure \ref{fig-sp-pi}.

This paper is organized as follows. In Section \ref{sec-relation} we give a proof of Theorem \ref{thm-2s-lc} and derive all the necessary tools. In Section \ref{sec-lc} we prove that the independence polynomials of all spiders are log-concave, present the symmetric function analogues of \eqref{eq-rec}, and use these two results to deduce the log-concavity of independence polynomials of
pineapple graphs.

\section{Schur-positivity and log-concavity}\label{sec-relation}

In this section, we shall introduce some basic definitions and prove Theorem \ref{thm-2s-lc}. After that, we give an explanation to our new method and derive the necessary tools for proving the main results in next section.

Let us first recall some basic terminology in the theory of symmetric functions.
Let $\mathbf{x} = \{x_1,x_2,\ldots\}$ be a set of countably infinite indeterminates. The algebra $\mathbb{Q}[[\mathbf{x}]]$ is defined as the commutative algebra of formal power series in these indeterminates over the rational field $\mathbb{Q}$. The \textit{algebra of symmetric functions} $\Lambda_{\mathbb{Q}}(\mathbf{x})$ is defined to be the subalgebra of $\mathbb{Q}[[\mathbf{x}]]$ consisting of formal power series $f(\mathbf{x})$ of bounded degree and satisfying
\[
f(\mathbf{x})=f(x_1,x_2,\ldots) = f(x_{\omega(1)},x_{\omega(2)},\ldots)
\]
for any permutation $\omega$ of positive integers.

The bases of $\Lambda_{\mathbb{Q}}(\mathbf{x})$ are naturally indexed by (integer) partitions. A \textit{partition} of $n$ is a sequence $\lambda = (\lambda_1,\ldots,\lambda_\ell)$ such that
\[
\lambda_1 \ge \lambda_2 \ge \cdots \ge \lambda_\ell > 0 \quad \mbox{and} \quad \lambda_1 + \lambda_2 + \cdots + \lambda_\ell = n.
\]
The number $\ell=\ell(\lambda)$ is called the \textit{length} of $\lambda$. By identifying $\lambda = (\lambda_1,\ldots,\lambda_\ell)$ with $(\lambda_1,\ldots,\lambda_\ell,0,0,\ldots)$, the \textit{monomial symmetric function} $m_{\lambda}(\mathbf{x})$ is defined as $\sum_{\alpha} x^{\alpha}$, where $\alpha$ ranges over all distinct permutations of $(\lambda_1,\ldots,\lambda_\ell,0,0,\ldots)$ and $x^{\alpha} = x_1^{\alpha_1}x_2^{\alpha_2}\cdots$ for $\alpha = (\alpha_1,\alpha_2,\ldots)$. The \textit{elementary symmetric function} $e_{\lambda}(\mathbf{x})$ is defined as
\begin{align*}
e_{k}(\mathbf{x})& = \sum_{1 \le j_1 < j_2 < \cdots < j_k} x_{j_1}x_{j_2} \cdots x_{j_k}, \mbox{ if } k\geq 1 \quad (e_0=1)\\
e_{\lambda}(\mathbf{x})& = e_{\lambda_1}(\mathbf{x})e_{\lambda_2}(\mathbf{x})\cdots e_{\lambda_\ell}(\mathbf{x}),  \mbox{ if } \lambda = (\lambda_1,\ldots,\lambda_\ell).
\end{align*}
The \textit{Schur function} $s_{\lambda}(\mathbf{x})$ can be defined by the dual Jacobi-Trudi identity
\[
s_{\lambda'}(\mathbf{x}) = \det (e_{\lambda_i-i+j}(\mathbf{x}))_{1 \le i,j \le \ell},
\]
where $\lambda'$ denotes the conjugate partition of $\lambda$, defined by $\lambda'_i = \#\{j \mid \lambda_j \ge i\}$ for $i \ge 1$, and $e_k(\mathbf{x})=0$ for $k<0$. It is well known that $\{m_{\lambda}(\mathbf{x})\},\{e_{\lambda}(\mathbf{x})\}$ and $\{s_{\lambda}(\mathbf{x})\}$ are bases of $\Lambda_{\mathbb{Q}}(\mathbf{x})$. We say that a symmetric function $f(\mathbf{x})$ is \textit{Schur-positive} or simply \textit{$s$-positive} if it
can be written as a nonnegative linear combination of Schur functions. As usual, given a basis $\{u_{\lambda}(\mathbf{x})\}$, we use $[u_{\lambda}]f(\mathbf{x})$ to denote the coefficient of $u_{\lambda}(\mathbf{x})$ in the expansion of $f(\mathbf{x})$ in terms of this basis.
For more information on symmetric functions, see \cite{Mac15,StaEC2}.


Now we turn to prove Theorem \ref{thm-2s-lc}. Recall that a polynomial $a_0 + a_1t + \cdots +a_dt^d$ is said to be log-concave if $a_j^2 \ge a_{j-1}a_{j+1}$ for $1 \le j \le d-1$, and it is said to be strongly log-concave if $a_m a_n \ge a_p a_q$ for all $1 \le p < m \le n < q \le d$ with $m+n = p+q$. It is clear that strong log-concavity implies log-concavity. By definition we have the following lemma.
\begin{lem}\label{lem-2tp}
A polynomial $a_0 + a_1t + \cdots + a_dt^d$ with nonnegative coefficients is strongly log-concave if and only if all the minor of order at most 2 of the Toeplitz matrix $(a_{j-i})_{i,j \ge 0}$ is nonnegative (namely, $(a_{j-i})_{i,j \ge 0}$ is 2-TP).
\end{lem}

Along the lines of Stanley's proof of Theorem \ref{thm-spos-rr}, we next give a proof of Theorem~\ref{thm-2s-lc}.

\begin{proof}[Proof of Theorem \ref{thm-2s-lc}.]
We give a cyclic proof. Let us first prove that (i) is equivalent to (iv). Suppose that $P(t) = \prod_{j=1}^d (1+\theta_j t)$ with $0 \neq \theta_j \in \mathbb{C}$.
Let $\mathbf{y}=\{y_1,y_2,\ldots\}$ and $\mathbf{\theta}=\{\theta_1,\theta_2,\ldots,\theta_d\}$.
By the Cauchy identity
\[
\prod_{i,j} (1+y_jx_i) = \sum_{\lambda} s_{\lambda'}(\mathbf{y})s_{\lambda}(\mathbf{x})
\]
we have
\begin{equation}\label{eq-fpx}
F_P(\mathbf{x}) = \sum_{\lambda} s_{\lambda'}(\mathbf{\theta})s_{\lambda}(\mathbf{x}),
\end{equation}
where $s_{\lambda'}(\mathbf{\theta})$ stands for the specialization of $s_{\lambda'}(\mathbf{y})$ by setting $y_1 = \theta_1, \ldots, y_d = \theta_d$ and $y_{l} = 0$ for $l\geq d+1$.
Note that the coefficient $a_i =[t^i]P(t)$ is equal to $e_i(\mathbf{\theta})$. Moreover, by the dual Jacobi-Trudi identity, all $s_{\lambda'}(\mathbf{\theta})$ with $\ell(\lambda) \le 2$ appear as minors of $(a_{j-i})_{i,j \ge 0}$ of order 1 or 2. Conversely, by the Littlewood-Richardson rule (see \cite{Mac15,StaEC2}), all minors of $(a_{j-i})_{i,j \ge 0}$ of order 1 or 2 are equal to $s_{\mu'/\nu'}(\mathbf{\theta})$ for some $\mu$ and $\nu$ with $\ell(\mu) \le 2$, which is a nonnegative linear combination of $s_{\lambda'}(\mathbf{\theta})$ with $\ell(\lambda) \le 2$. Hence $s_{\lambda'}(\mathbf{\theta}) \ge 0$ for all $\lambda$ with $\ell(\lambda) \le 2$ if and only if $(a_{j-i})_{i,j \ge 0}$ is 2-TP, and the latter is equivalent to the strong log-concavity of $P(t)$ by Lemma \ref{lem-2tp}. 

The implication ``(i) implies (ii)'' is obvious. We now prove that (ii) implies (iii). By \eqref{eq-fpx}, we find that
\[
[s_{(k,k)}]F_P(\mathbf{x}) = s_{(k,k)'}(\mathbf{\theta}) = \begin{vmatrix}
                                           e_k(\theta) & e_{k+1}(\theta) \\
                                           e_{k-1}(\theta) & e_k(\theta)
                                         \end{vmatrix} = a_k^2 - a_{k-1}a_{k+1} \ge 0.
\]
This proves (iii).

The final implication ``(iii) implies (iv)'' is also known. Indeed, if $a_i^2 - a_{i-1}a_{i+1} \ge 0$ for all $i \ge 1$, then
\[
\frac{a_1}{a_0} \ge \frac{a_2}{a_1} \ge \cdots \ge \frac{a_i}{a_{i-1}} \ge \frac{a_{i+1}}{a_i} \ge \cdots \ge \frac{a_d}{a_{d-1}},
\]
since all $a_i$ are positive. Hence $P(t)$ is strongly log-concave. 
\end{proof}

From now on we restrict our attention to the symmetric functions in the variables $\mathbf{x}=\{x_1,x_2,\ldots\}$.
We shall abbreviate $f(\mathbf{x})$ to $f$ if no confusion will result.
In order to apply the above theorem to independence polynomials, we shall first introduce the notions of multicolorings and clan graphs. Let $G$ be a graph with vertex set $V(G)$. Given a map $\alpha: V(G) \to \mathbb{N}$,
a \textit{multicoloring of type $\alpha$} is a map $\kappa: V(G) \to 2^{\mathbb{N}_+}$ such that $|\kappa(v)| = \alpha(v)$, where $2^{\mathbb{N}_+}$ denotes the set of all finite subsets of positive integers. A multicoloring is called \textit{proper} if $\kappa(u) \cap \kappa(v) = \emptyset$ for all $uv \in E(G)$. Stanley \cite{Sta98} defined
\[
X^{\alpha}_G = \sum x_1^{a_1}x_2^{a_2} \cdots,
\]
where the sum ranges over all multicolorings $\kappa$ of type $\alpha$ and $a_i$ is the number of vertices $v$ such that $i \in \kappa(v)$. In the special case of $\alpha = (1,1,\ldots,1)$, $X^{\alpha}_G$ reduces to $X_G$, the chromatic symmetric function of $G$ defined by Stanley \cite{Sta95}.

Given $\alpha$ as above, the clan graph $G^{\alpha}$ is defined to be the graph obtained from $G$ by replacing each $v_i$ by the complete graph $K_{\alpha(v_i)}$ while preserving the adjacency, namely, each vertex in $K_{\alpha(u)}$ is adjacent to each vertex in $K_{\alpha(v)}$ if $uv \in E(G)$. Stanley \cite{Sta98} noted that
\begin{equation}\label{eq-xg-alpha}
X_{G^{\alpha}} = X_G^{\alpha} \prod_{v \in V(G)}\alpha(v)!,
\end{equation}
and we shall also call $X_G^{\alpha}$ the normalized chromatic symmetric function of $X_{G^{\alpha}}$. The following corollary plays an essential role in the remaining part of this paper.

\begin{cor}\label{cor-2s-lc}
Let $G$ be a graph and
\begin{equation*}\label{eq-yg}
Y_G = \sum_{\alpha} X_G^{\alpha},
\end{equation*}
where the sum ranges over all such maps $\alpha$ from $V(G)$ to $\mathbb{N}$. Then the following conditions are equivalent:
\begin{itemize}
  \item $[s_{\lambda}] Y_G\ge 0$ for all partitions with $\ell(\lambda) \le 2$.
  \item $[s_{(k,k)}]Y_G \ge 0$ for all $k \ge 1$.
  \item The independence polynomial $I_G(t)$ is log-concave.
  \item The independence polynomial $I_G(t)$ is strongly log-concave.
\end{itemize}
\end{cor}

\begin{proof}
By noting that in each proper (multi)coloring of $G$ the vertices with the same color must be an independent set, Stanley \cite[Corollary 2.12]{Sta98} directly presented the following identity
\[
\sum_{\alpha:V \to \mathbb{N}} X_G^{\alpha} = \prod_i I_G(x_i).
\]
Then the corollary follows from Theorem \ref{thm-2s-lc}. %
\end{proof}

Motivated by Theorem \ref{thm-2s-lc} and Corollary \ref{cor-2s-lc}, we introduce the notion of 2-Schur-positivity.
We shall call a symmetric function $f$ \textit{2-Schur-positive} (or briefly \textit{2-$s$-positive}) if $[s_{\lambda}]f \ge 0$ for all $\lambda$ with $\ell(\lambda) \le 2$.
For convenience, we use $f \ge_{2s} 0$ to denote that $f$ is a {2-$s$-positive} symmetric function.
In particular, we write $f =_{2s} 0$ and call $f$ \textit{$2$-$s$-zero} if $[s_{\lambda}]f = 0$ for all $\lambda$ with $\ell(\lambda) \le 2$. In addition, if $f-g =_{2s} 0$ we also write $f =_{2s} g$.

Now by Corollary \ref{cor-2s-lc} the idea of our new method is rather simple: we just need to show that $Y_G$ is 2-$s$-positive by analyzing the 2-$s$-positivity of each $X_G^{\alpha}$. Further, Stanley noted that $X_{G + H} = X_GX_H$, where $G + H$ is the disjoint union of $G$ and $H$ (see \cite[Proposition 2.3]{Sta95}, and this identity also holds in general for $X_G^{\alpha}$). Hence the following lemma reduces this problem to consider the 2-$s$-positivity of each connected component of $X_G^{\alpha}$.

\begin{lem}\label{lem-2s-prod}
If $f$ and $g$ are 2-$s$-positive, then so is $fg$.
\end{lem}
\begin{proof}
By the Littlewood-Richardson rule, we see that
\[
s_{\lambda}s_{\mu} = \sum_{\nu} c_{\lambda\mu}^{\nu} s_{\nu},
\]
where $c_{\lambda\mu}^{\nu} \ge 0$ and $c_{\lambda\mu}^{\nu} = 0$ if $\ell(\nu) < \ell(\lambda)$ or $\ell(\nu) < \ell(\mu)$. The desired result immediately follows.
 \end{proof}

Next, for calculating the coefficients of $s_{\lambda}$'s in $X_G$ with $\ell(\lambda) \le 2$, we have a simple formula that involves stable partitions of $V(G)$. Let $B = \{B_1,B_2,\ldots,B_k\}$ be a set partition of $V(G)$. Then $B$ is said to be a \textit{stable partition} of $G$ if each block $B_i$ is an independent set, and the \textit{type} of $B$ is defined to be the partition formed by $|B_1|,|B_2|,\ldots,|B_k|$ in weakly decreasing order. A stable partition is called \textit{semi-ordered} if the blocks of the same cardinality are ordered. Then the formula is stated as follows, which could be deduced from a result of Wang and Wang \cite[Theorem 3.1]{WW20}.
%
%
%

\begin{prop}\label{prop-2s-coe}
Let $G$ be a graph with $|V(G)| = n$. Let $(k,l)$ be a partition of $n$ with $k \ge l \ge 1$ and let $\tilde{a}_{(k,l)}$ denote the number of semi-ordered stable partition of $G$ with type $(k,l)$. Then
\[
[s_{(k,l)}] X_G = \tilde{a}_{(k,l)} - \tilde{a}_{(k+1,l-1)}, \,\, [s_{(n)}] X_G = \tilde{a}_{(n)} = \begin{cases}
                          1, & \mbox{if $G$ consists of $n$ isolated vertices;} \\
                          0, & \mbox{otherwise}.
                        \end{cases}
\]
\end{prop}

One may observe that the coefficient of $s_{(n)}$ in each $X_G^{\alpha}$ is always nonnegative, and hence from now on we shall only focus on the coefficients of partitions with length 2. Moreover, if a graph has a stable partition of length 2, then it must be a bipartite graph. More precisely, we have the following corollary, which also implies that most terms in $Y_G$ are 2-$s$-zero and have no influence on the log-concavity of $I_G(t)$.

\begin{cor}\label{cor-2s-con-bipar}
For a connected bipartite graph $G$, $X_G$ is 2-$s$-positive if and only if its unique bipartition is balanced, i.e., it has type $(k,l)$ with $l \le k \le l+1$. Equivalently, $X_G$ is not 2-$s$-positive if and only if its bipartition is of type $(k,l)$ with $k \ge l+2$. In particular, $X_G =_{2s} 0$ for all non-bipartite graph $G$.
\end{cor}
\begin{proof}
Note that any connected bipartite graph has a unique bipartition. If $G$ consists of only one vertex, then the result follows trivially. Otherwise, if the bipartition of $G$ has type $(k,l)$ with $k \ge l+2$, then by Proposition \ref{prop-2s-coe}, $[s_{(k-1,l+1)}]X_G=\tilde{a}_{(k-1,l+1)} - \tilde{a}_{(k,l)} = -1$, where $(k-1,l+1)$ is a valid partition since $k-1 \ge l+1$.

Conversely, if $[s_{(m,n)}]X_G < 0$ for some $m \ge n \ge 1$, then $\tilde{a}_{(m+1,n-1)} > 0$, which implies that there exists a stable partition of type $(m+1,n-1)$, where $m+1 \ge (n-1) + 2$ since $m \ge n$.

Finally, if a graph is not bipartite, then it contains an odd cycle. One can observe that the vertices of any odd cycle cannot be divided into two independent subsets, which means that $G$ cannot have a stable partition of length 2 and hence $X_G$ is 2-$s$-zero.
\end{proof}

As stated before, Corollary \ref{cor-2s-lc} and Lemma \ref{lem-2s-prod} allow us to consider the 2-$s$-positivity of each connected component of each $X_G^{\alpha}$, and now Proposition \ref{prop-2s-coe} and Corollary \ref{cor-2s-con-bipar} provide the formula for calculating the coefficients of $s_{\lambda}$'s with $\ell(\lambda) \le 2$. Hence in order to prove that $Y_G$ is 2-$s$-positive we just need to show that all the (possible) negative terms will be eliminated in the summation $\sum_{\alpha}X_G^{\alpha}$.

To get some hints of the efficacy of this approach, let us first note a new proof of the log-concavity of independence polynomials of claw-free graphs, which was first proved by Hamidoune \cite{Ham90}.
If a bipartite graph is claw-free, i.e., it has no induced subgraph isomorphic to $K_{1,3}$, then it consists of paths and even cycles, which are both $2$-$s$-positive by direct calculation. (In fact, all paths and cycles are $e$-positive by Stanley \cite[Propositions 5.3 and 5.4]{Sta95}.) Hence by Corollary \ref{cor-2s-con-bipar} any claw-free graph $G$ has $2$-$s$-positive chromatic symmetric function, which yields the $2$-$s$-positivity of $Y_G$ since all clan graphs $G^{\alpha}$ are also claw-free. Then by Corollary \ref{cor-2s-lc}, we obtain the log-concavity of independence polynomials of claw-free graphs. Hamidoune's log-concavity result was strengthened to real-rootedness by Chudnovsky and Seymour \cite{CS07}, which is closely related to Stanley's conjecture on $s$-positivity of all claw-free graphs in view of Theorem \ref{thm-spos-rr}.

\section{Log-concavity of independence polynomials}\label{sec-lc}

This section is devoted to presenting our main results. In the first subsection, we shall prove that independence polynomials of all spiders are log-concave. In the second subsection, we shall give two symmetric function analogues of \eqref{eq-rec}, and apply one of them to deduce the log-concavity of independence polynomials of pineapple graphs.

\subsection{Log-concavity of independence polynomials of spiders}

In this subsection we shall  prove the log-concavity of independence polynomials of spiders. For convenience, we use $S(\lambda)$ to represent the spider whose leg lengths form an integer partition   $\lambda=(\lambda_1,\lambda_2,\ldots,\lambda_{\ell(\lambda)})$. In addition, we label the torso of the spider as $v_0$, and label the vertex adjacent to $v_0$ in the $i$-th leg as $v_i$ for each $1 \le i \le \ell(\lambda)$, see Figure \ref{fig-spider}.
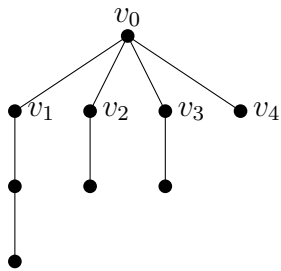
\begin{figure}[htbp]
\centering
\begin{tikzpicture}[scale = 1]

\fill (-6,0) circle (0.5ex);
\fill (-3,0) circle (0.5ex);
\fill (-4,0) circle (0.5ex);
\fill (-4,-1) circle (0.5ex);

\fill (-4.5,1) circle (0.5ex);
\fill (-5,0) circle (0.5ex);
\fill (-6,-1) circle (0.5ex);
\fill (-6,-2) circle (0.5ex);
\fill (-5,-1) circle (0.5ex);
\draw (-4.5,1) -- (-6,0) -- (-6,-1) -- (-6,-2);
\draw (-4.5,1) -- (-5,0) -- (-5,-1);
\draw (-4.5,1) -- (-4,0) -- (-4,-1);
\draw (-4.5,1) -- (-3,0);
\node at (-4.5,1.25) {$v_0$};
\node at (-5.65,0) {$v_1$};
\node at (-4.65,0) {$v_2$};
\node at (-3.65,0) {$v_3$};
\node at (-2.65,0) {$v_4$};
\end{tikzpicture}\caption{Spider $S(3,2,2,1)$}\label{fig-spider}
\end{figure}


\begin{thm}\label{thm-spi-lc}
Let $\lambda$ be a partition and $S(\lambda)$ be the corresponding spider. Then $Y_{S(\lambda)}$ is 2-$s$-positive, or equivalently, the independence polynomial $I_{S(\lambda)}(t)$ is strongly log-concave.
\end{thm}

Let us first introduce the outline of this proof. At the first step, we need to analyze all $X_{S(\lambda)}^{\alpha}$ and find out the maps $\alpha$ with $X_{S(\lambda)}^{\alpha}$ being (possibly) not 2-$s$-positive. (This is different from the case of claw-free graphs $G$ where all $X_{G}^{\alpha}$ are 2-$s$-positive.) Then we shall define an injection $\phi$ from the set of such map $\alpha$'s to the set of remaining maps. Finally we show that $X_{S(\lambda)}^{\alpha}+X_{S(\lambda)}^{\phi(\alpha)} \ge_{2s} 0$, which means all the negative terms can be eliminated and hence $Y_{S(\lambda)} \ge_{2s} 0$.


\begin{proof}[Proof of Theorem \ref{thm-spi-lc}.]
By the idea sketched above, our proof will be accomplished by the following three steps.

\textbf{Step 1: analyzing all $X_{S(\lambda)}^{\alpha}$}

We list all the cases for $S(\lambda)^{\alpha}$, where we abbreviate $\alpha(v_i)$ as $\alpha_i$. 
\begin{itemize}
  \item[(1)] $\alpha_0 \ge 3$: 
      $X_{S(\lambda)^{\alpha}} =_{2s} 0$ since there exists a triangle in $K_{\alpha_0}$ and hence the graph is not bipartite (see Corollary \ref{cor-2s-con-bipar}).
  \item[(2)] $\alpha_0 = 2$ and 
      $\alpha_i \ge 1$ for some $1 \le i \le \ell(\lambda)$: $X_{S(\lambda)^{\alpha}} =_{2s} 0$ for the same reason as (1), and here the triangle is formed by the vertices in $K_{\alpha_0}$ and $K_{\alpha_i}$.
  \item[(3)] $\alpha_0 =2$ 
      and $\alpha_i =0$ for all $1 \le i \le \ell(\lambda)$: $X_{S(\lambda)^{\alpha}} \ge_{2s} 0$. Indeed, if $S(\lambda)^{\alpha}$ is bipartite in this case, then all of its components are paths or isolated points. Hence $X_{S(\lambda)^{\alpha}}$ is 2-$s$-positive by Lemma \ref{lem-2s-prod} and Corollary \ref{cor-2s-con-bipar}.
  \item[(4)] $\alpha_0 = 1$ 
      and $\alpha_i =0$ for all $1 \le i \le \ell(\lambda)$: $X_{S(\lambda)^{\alpha}} \ge_{2s} 0$ for the same reason as (3).
  \item[(5)] $\alpha_0 =1$ 
      and $\alpha_i \ge 1$ for some $1 \le i \le \ell(\lambda)$: we have the following subcases:
      \begin{itemize}
        \item[(5.1)] $\alpha_i \ge 2$ for some $1 \le i \le \ell(\lambda)$: $X_{S(\lambda)^{\alpha}} =_{2s} 0$ for the same reason as (1).
        \item[(5.2)] $\alpha_i \le 1$ for all $1 \le i \le \ell(\lambda)$: this is the \emph{only} possible case for $X_{S(\lambda)^{\alpha}}$ to be non-2-$s$-positive.
      \end{itemize}
  \item[(6)] $\alpha_0 = 0$: 
      $X_{S(\lambda)^{\alpha}} \ge_{2s} 0$ for the same reason as (3).
\end{itemize}

Next, we investigate the $S(\lambda)^{\alpha}$'s of (5.2) in more details. We take the component containing $v_0$ and call it $C_0 = C_0(\lambda,\alpha)$, which is actually a smaller spider and hence bipartite. Moreover, the complement graph of $C_0$ in $S(\lambda)^{\alpha}$ is denoted by $\overline{C_0}$. Then we have the following observation:
\begin{obs}\label{obs-2s}
If $X_{S(\lambda)^{\alpha}}$ is non-2-$s$-positive, then $C_0$ is non-2-$s$-positive.
\end{obs}
In fact, similar to Case (3), all the components of $\overline{C_0}$ are clan graphs of paths, which are either 2-$s$-positive or not bipartite. Hence if $C_0$ is 2-$s$-positive, then $X_{S(\lambda)^{\alpha}} = X_{C_0}X_{\overline{C_0}}$ must be 2-$s$-positive by Lemma \ref{lem-2s-prod}. This proves Observation \ref{obs-2s}.

Now we divide (5.2) into two subcases:
\begin{itemize}[leftmargin=1.5cm]
 \item[(5.2.1)] $C_0$ is 2-$s$-positive. In this subcase $X_{S(\lambda)^{\alpha}} \ge_{2s} 0$.
 \item[(5.2.2)] $C_0$ is non-2-$s$-positive. Then, by Corollary \ref{cor-2s-con-bipar}, its bipartition is of type $(k,l)$ with $k-l \ge 2$.
\end{itemize}
In the following by abuse of notation we will use \textbf{(5.2.2)} (resp. \textbf{(6)}) to represent the set of those maps $\alpha: V(S(\lambda)) \to \mathbb{N}$ satisfying the condition of (5.2.2) (resp. (6)).

\textbf{Step 2: defining $\phi:\textbf{(5.2.2)}\rightarrow \textbf{(6)}$ and proving that it is an injection}

Now we shall give an injection $\phi$ from \textbf{(5.2.2)} to \textbf{(6)}. Denote $C_0$ by $S(\beta)$ as a subgraph of $S(\lambda)$, where $\beta$ is a weak composition satisfying $0 \le \beta_i \le \lambda_i$ for all $1 \le i \le \ell(\lambda)$.
(Here we do not require that $\beta$ is a partition. Hence we keep the order of the original legs of $S(\lambda)$, as well as the order of vertices on each leg.)
Denote the set of legs of odd length and even length of $S(\beta)$ by $\{L(\beta_{i_1}),\ldots,L(\beta_{i_a})\}$ and $\{L(\beta_{j_1}),\ldots,L(\beta_{j_b})\}$ respectively. Then one can verify that the sizes of the two blocks of the bipartition of $C_0$ are
{\small \[
\frac{\beta_{i_1}+1}{2}+\cdots+\frac{\beta_{i_a}+1}{2}+\frac{\beta_{j_1}}{2}+ \cdots + \frac{\beta_{j_b}}{2} \quad \mbox{and} \quad 1+\frac{\beta_{i_1}-1}{2}+\cdots+\frac{\beta_{i_a}-1}{2}+\frac{\beta_{j_1}}{2}+ \cdots + \frac{\beta_{j_b}}{2}.
\]}

\noindent Therefore, if the difference of these two numbers is at least 2, then there must exist at least 3 legs of odd length in $C_0 = S(\beta)$, i.e., $a \ge 3$. Further, for notational convenience, we shall use the following symbols:
$$\beta_{odd}^+ = \frac{\beta_{i_1}+1}{2}+\cdots+\frac{\beta_{i_a}+1}{2}, \beta_{odd}^- =\frac{\beta_{i_1}-1}{2}+\cdots+\frac{\beta_{i_a}-1}{2}, \beta_{even} = \frac{\beta_{j_1}}{2}+ \cdots + \frac{\beta_{j_b}}{2}. $$

Take the first leg (using the order in $\beta = (\beta_1,\ldots,\beta_{\ell(\lambda)})$) in the set of legs with the smallest odd length. Let this leg be $L(\beta_{i_k}) = \{v_{q_1},\ldots,v_{q_{\beta_{i_k}}}\}$. Then we define $\phi$ as follows:
\begin{itemize}
  \item[(i)] If $\beta_{i_k} = 1$, then
\[
(\phi(\alpha))_i = \begin{cases}
                     0, & \mbox{if } i=0; \\
                     2, & \mbox{if } i=q_1=q_{\beta_{i_k}}; \\
                     \alpha_i, & \mbox{otherwise}.
                   \end{cases}
\]
In terms of graphs, we just delete the vertex $v_0$ and replace the vertex $v_{q_1}$ by a $K_2$ in $S(\lambda)^{\alpha}$ to obtain $S(\lambda)^{\phi(\alpha)}$, see Figure \ref{fig-phi(i)}.
\begin{figure}[htbp]
\centering
\begin{tikzpicture}[scale = 1.2]
\fill (0,1) node (v1) {} circle (0.5ex);
\fill (0,0) node (v7) {} circle (0.5ex);
\fill (-1,0) node (v9) {} circle (0.5ex);
\fill (-1,-1) node (v9) {} circle (0.5ex);
\fill (-1,-2) node (v9) {} circle (0.5ex);
\fill (1,0) node (v10) {} circle (0.5ex);
\fill (1,-1) node (v10) {} circle (0.5ex);
\fill (1,-2) node (v10) {} circle (0.5ex);
\draw (0,1) -- (-1,0) -- (-1,-1) -- (-1,-2);
\draw (0,1) -- (0,0);
\draw (0,1) -- (1,0) -- (1,-1) -- (1,-2);
\fill (0,-2) node (v10) {} circle (0.5ex);

\node at (2,0) {$\overset{\phi}{\rightarrow}$};

\fill (4.25,0) node (v7) {} circle (0.5ex);
\fill (3.75,0) node (v7) {} circle (0.5ex);
\draw (4.25,0) -- (3.75,0);
\fill (5,0) node (v9) {} circle (0.5ex);
\fill (3,-1) node (v9) {} circle (0.5ex);
\fill (3,-2) node (v9) {} circle (0.5ex);
\fill (3,0) node (v10) {} circle (0.5ex);
\fill (5,-1) node (v10) {} circle (0.5ex);
\fill (5,-2) node (v10) {} circle (0.5ex);
\draw (3,0) -- (3,-1) -- (3,-2);
\draw (5,0) -- (5,-1) -- (5,-2);
\fill (4,-2) node (v10) {} circle (0.5ex);
\node at (0,1.5) {\footnotesize $v_0$};
\node at (-1,0.5) {\footnotesize $v_{\beta_{i_{k-1}}}$};
\node at (1,0.5) {\footnotesize $v_{\beta_{i_{k+1}}}$};
\node at (3,0.5) {\footnotesize $v_{\beta_{i_{k-1}}}$};
\node at (0,-0.5) {\footnotesize $v_{\beta_{i_k}}$};
\node at (3.75,0.5) {\footnotesize $v_{\beta_{i_k}}$};
\node at (4.25,0.5) {\footnotesize $v'_{\beta_{i_k}}$};
\node at (5,0.5) {\footnotesize $v_{\beta_{i_{k+1}}}$};
\end{tikzpicture}
    \caption{An example for Case (i) of $\phi$}\label{fig-phi(i)}
\end{figure}
  \item[(ii)] If $\beta_{i_k} \ge 3$, then take the first leg in the other legs of odd length, say $L(\beta_{i_l}) = \{v_{r_1},\ldots,v_{r_{\beta_{i_l}}}\}$. Note that $\beta_{i_l} \ge \beta_{i_k}$ by our choice. Let
\end{itemize}
\[
(\phi(\alpha))_i = \begin{cases}
                     0, & \mbox{if } i=0; \\
                     2, & \mbox{if } i=q_1,q_3,\ldots,q_{\beta_{i_k}}; \\
                     0, & \mbox{if } i=q_2,q_4,\ldots,q_{\beta_{i_k}-1}; \\
                     2, & \mbox{if } i=r_1,r_3,\ldots,r_{\beta_{i_k}-2}; \\
                     0, & \mbox{if } i=r_2,r_4,\ldots,r_{\beta_{i_k}-1}; \\
                     \alpha_i, & \mbox{otherwise}.
                   \end{cases}
\]
In other words, we delete the vertex $v_0$, replace $L(\beta_{i_k})$ by $(\beta_{i_k}+1)/2$ $K_2$'s, and replace the first $\beta_{i_k}-1$ vertices in $L(\beta_{i_l})$ by $(\beta_{i_k}-1)/2$ $K_2$'s. See Figure \ref{fig-phi(ii)} for an illustration, where $i_l=i_{k+1}$. We should remark that the operation of $L(\beta(i_l))$ labels the length of $L(\beta_{i_k})$, which plays an essential role for proving the injectivity of $\phi$.
\begin{figure}[htbp]
\centering
\begin{tikzpicture}[scale=1.2]
\fill (0,1) node (v1) {} circle (0.5ex);
\fill (0,0) node (v7) {} circle (0.5ex);
\fill (-1,0) node (v9) {} circle (0.5ex);
\fill (-1,-1) node (v9) {} circle (0.5ex);
\fill (-1,-2) node (v9) {} circle (0.5ex);
\fill (1,0) node (v10) {} circle (0.5ex);
\fill (1,-1) node (v10) {} circle (0.5ex);
\fill (1,-2) node (v10) {} circle (0.5ex);
\fill (1,-3) node (v10) {} circle (0.5ex);
\fill (1,-4) node (v10) {} circle (0.5ex);
\draw (0,1) -- (-1,0) -- (-1,-1) -- (-1,-2);
\draw (0,1) -- (0,0) -- (0,-1) -- (0,-2);
\draw (0,1) -- (1,0) -- (1,-1) -- (1,-2) -- (1,-3) -- (1,-4);
\fill (0,-1) node (v10) {} circle (0.5ex);
\fill (0,-2) node (v10) {} circle (0.5ex);
\fill (-0.75,-4) node (v10) {} circle (0.5ex);
\fill (-1.25,-4) node (v10) {} circle (0.5ex);
\draw (-0.75,-4) -- (-1.25,-4);
\fill (0,-4) node (v7) {} circle (0.5ex);

\node at (2,0) {$\overset{\phi}{\rightarrow}$};

\fill (4.25,0) node (v7) {} circle (0.5ex);
\fill (3.75,0) node (v7) {} circle (0.5ex);
\draw (4.25,0) -- (3.75,0);
\fill (5,0) node (v9) {} circle (0.5ex);
\fill (2.75,-2) node (v9) {} circle (0.5ex);
\fill (3.25,-2) node (v9) {} circle (0.5ex);
\fill (2.75,0) node (v10) {} circle (0.5ex);
\fill (3.25,0) node (v10) {} circle (0.5ex);
\fill (5,-1) node (v10) {} circle (0.5ex);
\fill (5,-2) node (v10) {} circle (0.5ex);
\draw (2.75,0) -- (3.25,0);
\draw (2.75,-2) -- (3.25,-2);
\draw (5,0) -- (5,-1) -- (5,-2) -- (5,-3) -- (5,-4);
\fill (4,-2) node (v10) {} circle (0.5ex);
\fill (2.75,-4) node (v10) {} circle (0.5ex);
\fill (3.25,-4) node (v10) {} circle (0.5ex);
\draw (2.75,-4) -- (3.25,-4);
\fill (4,-4) node (v7) {} circle (0.5ex);
\fill (5,-3) node (v10) {} circle (0.5ex);
\fill (5,-4) node (v10) {} circle (0.5ex);
%
\node at (-1,0.5) {\footnotesize$v_{\beta_{i_{k}}}$};
\node at (0,0.5) {\footnotesize$v_{\beta_{i_{k+1}}}$};
\node at (1,0.5) {\footnotesize$v_{\beta_{i_{k+2}}}$};
\node at (2.65,0.5) {\footnotesize $v_{\beta_{i_{k}}}$};
\node at (3.15,0.5) {\footnotesize$v_{\beta_{i_{k}}}'$};
\node at (3.65,0.5) {\footnotesize$v_{\beta_{i_{k+1}}}$};
\node at (4.35,0.5) {\footnotesize$v_{\beta_{i_{k+1}}}'$};
\node at (5,0.5) {\footnotesize$v_{\beta_{i_{k+2}}}$};
\end{tikzpicture}
    \caption{An example for Case (ii) of $\phi$}\label{fig-phi(ii)}
\end{figure}

It can be easily verified that in both cases $\phi$ preserves the number of vertices of $X_{S(\lambda)^{\alpha}}$ and $\phi$ is exactly a map from \textbf{(5.2.2)} to \textbf{(6)}. It remains to show that $\phi$ is actually an injection.

Note that case (i) and case (ii) are distinguished by the number of indices $1 \le i \le \ell(\lambda)$ with $\alpha_i = \alpha(v_i) = 2$. In case (i), there is only one such index and hence in order to get the preimage we just replace the corresponding $K_2$ by one vertex $v_i$ and add $v_0$ back. In case (ii), there are two such indices, say $i_k,i_l$, and there are consecutive $K_2$'s starting from $\alpha_{i_k} = \alpha(v_{i_k}) = 2$ and $\alpha_{i_l} = \alpha(v_{i_l}) =2$ in the two corresponding legs respectively. These consecutive $K_2$'s will appear as subsequences of the form $202020\cdots$, where $2$ represents $K_2$ and $0$ represents the empty graph. (In each leg we read the graphs from top to bottom.) The numbers of such consecutive $K_2$'s are actually different in these two legs by the definition of $\phi$. If we denote the smaller number by $t$, then there will be a subsequence $2020\cdots 201$ of length $2t+1$ in this leg and a subsequence $2020\cdots 202$ of length $2t+1$ in the other leg, where the numbers of $K_2$'s are $t$ and $t+1$ respectively. In order to construct the preimage we just need to replace these two subsequences by $11\cdots 1$ of length $2t+1$.


\textbf{Step 3: showing that all negative terms can be eliminated}

As the final step, we give the proof of
\begin{equation*}
X^{\alpha}_{S(\lambda)}+X^{\phi(\alpha)}_{S(\lambda)} \ge_{2s} 0
\end{equation*}
for each $\alpha$ satisfying the condition of (5.2.2).
Note that the $\alpha_i$'s are fixed under $\phi$ for $v_i \in \overline{C_0}$, and these $\alpha_i$'s induce clan graphs of paths as mentioned in the explanation for Observation \ref{obs-2s}. To be precise, let $\alpha|_{C_0}$ be the restriction of $\alpha$ to $C_0$, i.e., the vector obtained from $\alpha$ by replacing all $\alpha_i$ with 0 for $v_i \not\in C_0$ and keeping the $\alpha_i$ for $v_i \in C_0$. The notation $\alpha|_{\overline{C_0}}$ is defined similarly. Then $X^{\alpha|_{\overline{C_0}}}_{S(\lambda)} = X^{\phi(\alpha)|_{\overline{C_0}}}_{S(\lambda)}$ and they are both 2-$s$-positive. Further, by \eqref{eq-xg-alpha} we have
\begin{align*}
X^{\alpha}_{S(\lambda)} = X^{\alpha|_{C_0}}_{S(\lambda)} \cdot X^{\alpha|_{\overline{C_0}}}_{S(\lambda)} \quad \mbox{and} \quad X^{\phi(\alpha)}_{S(\lambda)} = X^{\phi(\alpha)|_{C_0}}_{S(\lambda)} \cdot X^{\phi(\alpha)|_{\overline{C_0}}}_{S(\lambda)}.
\end{align*}
Hence in order to prove $X^{\alpha}_{S(\lambda)}+X^{\phi(\alpha)}_{S(\lambda)} \ge_{2s} 0$, it suffices to show that  \begin{equation}\label{eq-phi-2spos}
X^{\alpha|_{C_0}}_{S(\lambda)}+X^{\phi(\alpha)|_{C_0}}_{S(\lambda)} \ge_{2s} 0.
\end{equation}

Note that $S(\lambda)^{\alpha|_{C_0}}$ is exactly the spider $C_0$.
In $C_0$, there is only one stable partition of length 2, and its type is $\left(\beta_{odd}^+ + \beta_{even}, \beta_{odd}^- + \beta_{even} + 1\right)$.
\noindent Hence
\[
X_{S(\lambda)}^{\alpha|_{C_0}} = X_{S(\lambda)^{\alpha|_{C_0}}} =_{2s}
s_{\left(\beta_{odd}^+ + \beta_{even}, \beta_{odd}^- + \beta_{even} + 1\right)} -
s_{\left(\beta_{odd}^+ + \beta_{even}-1, \beta_{odd}^- + \beta_{even} + 2\right)}.
\]
Here is a key observation: since $X^{\phi(\alpha)|_{C_0}}_{S(\lambda)}$ is already 2-$s$-positive by the definition of $\phi$, in order to prove \eqref{eq-phi-2spos} we only need to show
\[
\left[s_{\left(\beta_{odd}^+ + \beta_{even}-1, \beta_{odd}^- + \beta_{even} + 2\right)}\right]
X_{S(\lambda)}^{\phi(\alpha)|_{C_0}} \ge 1.
\]

Now we have two cases for calculating this coefficient in $X_{S(\lambda)^{\phi(\alpha)|_{C_0}}}$, according to the value of $\beta_{i_k}$ in the definition of the map $\phi$:

(i) $\beta_{i_k} = 1$. In this case there are $2^{b+1}$ ways to obtain a semi-ordered stable partition of type
$\left(\beta_{odd}^+ + \beta_{even}, \beta_{odd}^- + \beta_{even} + 1\right)$.
\noindent Indeed, to obtain the first block of this stable partition, we have to choose exactly one vertex $v_{\beta_{i_k}}$ or $v'_{\beta_{i_k}}$ in the $K_2$ corresponding to $L(\beta_{i_k})$, choose $(\beta_{i_t}+1)/2$ vertices in each leg $L(\beta_{i_t})$ of odd length with $1 \le t \le a\, (t \neq k)$ in a unique way, and choose $\beta_{j_t}/2$ vertices in each leg $L(\beta_{i_t})$ of even length with $1 \le t \le b$ in two ways. Thus,  the number of stable partitions of type  $\left(\beta_{odd}^+ + \beta_{even}, \beta_{odd}^- + \beta_{even} + 1\right)$ is $2 \cdot 2^b = 2^{b+1}$. Similarly, there are $(a-1)2^{b+1}$ ways to obtain a semi-ordered stable partition with type $\left(\beta_{odd}^+ + \beta_{even}-1, \beta_{odd}^- + \beta_{even} + 2\right)$  by exchanging the biparition of one leg $L(\beta_{i_t})$ of odd length for some $1 \le t \le a\,(t \neq k)$ .
By Proposition \ref{prop-2s-coe} we get
\begin{align*}
&\left[s_{\left(\beta_{odd}^+ + \beta_{even}-1, \beta_{odd}^- + \beta_{even} + 2\right)}\right]X^{\phi(\alpha)|_{C_0}}_{S(\lambda)} \\
 =&
\left[s_{\left(\beta_{odd}^+ + \beta_{even}-1, \beta_{odd}^- + \beta_{even} + 2\right)}\right]\frac{1}{2}X_{S(\lambda)^{\phi(\alpha)|_{C_0}}} \\
=& (a-2)2^{b} \ge 1,
\end{align*}
%
which follows from $a \ge 3$ and $b \ge 0$.

(ii) $\beta_{i_k} \ge 3$. The calculation in this case is exactly analogous to case (i) and we shall only list the data. We have
  $2^{(\beta_{i_k}+1)/2+(\beta_{i_k}-1)/2+b} = 2^{\beta_{i_k}+b}$  ways to obtain a semi-ordered stable partition of type $\left(\beta_{odd}^+ + \beta_{even}, \beta_{odd}^- + \beta_{even} + 1\right)$.
Moreover, we have $(a-1)2^{\beta_{i_k}+b}$ 
semi-ordered stable partitions of type $\left(\beta_{odd}^+ + \beta_{even}-1, \beta_{odd}^- + \beta_{even} + 2\right)$.
\noindent
Hence
\begin{align*}
&\left[s_{\left(\beta_{odd}^+ + \beta_{even}-1, \beta_{odd}^- + \beta_{even} + 2\right)}\right]X^{\phi(\alpha)|_{C_0}}_{S(\lambda)} \\
 =&
\left[s_{\left(\beta_{odd}^+ + \beta_{even}-1, \beta_{odd}^- + \beta_{even} + 2\right)}\right]\frac{1}{2^{\beta_{i_k}}}X_{S(\lambda)^{\phi(\alpha)|_{C_0}}} \\
=& (a-2)2^{b} \ge 1,
\end{align*}
which also follows from $a \ge 3$ and $b \ge 0$.

Combining the above steps, we complete the proof.
\end{proof}

\subsection{Symmetric function analogues of the recurrence formula}

In this subsection we shall present two symmetric function analogues of the recurrence formula \eqref{eq-rec}. Further, we apply these results to prove the log-concavity of independence polynomials of pineapple graphs.

\begin{prop}\label{prop-rec}
Let $v$ be a vertex in $G$
and $G-v$, $N[v] = \{v\} \cup N(v)$ be defined as in \eqref{eq-rec}. Then we have
\[
Y_G = Y_{G-v} + (Y_v-1)Y_{G-N[v]} + \sum_{\alpha}X_G^{\alpha},
\]
where $\alpha$ ranges over all maps from $V(G)$ to $\mathbb{N}$ such that $\alpha(v) \ge 1$ and $\alpha(u) \ge 1$ for some $u \in N(v)$.
\end{prop}

\begin{proof}
In order to prove this proposition we only need to divide the $\alpha$'s into the following three classes:
\begin{itemize}
    \item[(1)] $\alpha(v) = 0$;
    \item[(2)] $\alpha(v) \ge 1$ and $\alpha(u) = 0$ for all $u \in N(v)$;
    \item[(3)] $\alpha(v) \ge 1$ and $\alpha(u) \ge 1$ for some $u \in N(v)$.
\end{itemize}
Then it is straightforward to verify that the sum of all $X_G^{\alpha}$ in (1) is exactly $Y_{G-v}$. Similarly the $\alpha$'s in (2) correspond to $(Y_v-1)Y_{G-N[v]}$ since $X_v^{\alpha(v)} = 1$ for $\alpha(v) = 0$. Then the proof follows.
\end{proof}

\noindent\textbf{Remark.} The ``error term''  $\sum_{\alpha} X_G^{\alpha}$ in this result to some extent illustrates why the real-rootedness (or log-concavity) of $G-v$ and $G-N[v]$ cannot directly yields that of $G$.

In the special case of clan graphs we have a slightly different result.
\begin{prop}\label{prop-rec-clan}
Let $v$ be a leaf vertex of $G$ and $uv$ be the corresponding leaf edge. Let $G_{v}^n$ be the clan graph of $G$ obtained by replacing $v$ by a complete graph $K_n (n \ge 2)$. Then
\[
Y_{G_v^n} = Y_{G-v} + (Y_{K_{n}}-1)Y_{G-\{u,v\}} + \sum_{\alpha}X_{G_v^n}^{\alpha},
\]
where $\alpha$ ranges over all maps from $V(G)$ to $\mathbb{N}$ such that $\alpha(u) \ge 1$ and $\alpha(v_i) \ge 1$ for some $1 \le i \le n$.
\end{prop}
\begin{proof}
Here we regard the vertices in $K_n$ as a whole and label them as $\{v_1,v_2,\ldots,v_n\}$. The we divide the $\alpha$'s according to the following three cases:
\begin{itemize}
    \item[(1)] $\alpha(v_i) = 0$ for any $1 \le i \le n$;
    \item[(2)] $\alpha(u) = 0$ and $\alpha(v_i) \ge 1$ for some $1 \le i \le n$;
    \item[(3)] $\alpha(u) \ge 1$ and $\alpha(v_i) \ge 1$ for some $1 \le i \le n$.
\end{itemize}
The remaining part of the proof is exactly similar to that of Proposition \ref{prop-rec}.
\end{proof}

In general the 2-$s$-positivity of $Y_{G-v}$ and $Y_{G-\{u,v\}}$ do not imply the 2-$s$-positivity of $Y_{G_v^n}$ since $\sum_{\alpha}X_{G_v^n}^{\alpha}$ in Proposition \ref{prop-rec-clan} may not be 2-$s$-positive. However, in some special cases we are able to eliminate the negative terms in these $X_{G_v^n}^{\alpha}$'s. As an example, we proceed to show that each pineapple graph has a log-concave independence polynomial. Precisely, for any integer $n \ge 2$ and partition $\lambda$, the pineapple graph $Pi(n,\lambda)$ is defined to be the graph obtained by identifying a vertex of the complete graph $K_n$ with the torso of the spider $S(\lambda)$. In addition, we label the torso as $u$ and the other vertices in the complete graph $K_n$ as $v_1,\ldots,v_{n-1}$. as illustrated in Figure \ref{fig-pineapple}. 
Then we have the following result.
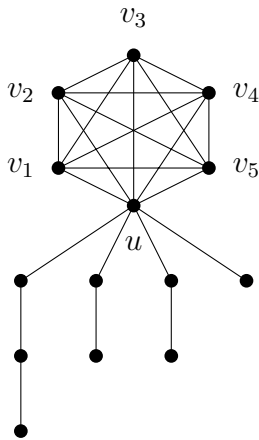
\begin{figure}[htbp]
\centering
\begin{tikzpicture}[scale = 1]

\fill (-6,0) circle (0.5ex);
\fill (-3,0) circle (0.5ex);
\fill (-4,0) circle (0.5ex);
\fill (-4,-1) circle (0.5ex);

\fill (-4.5,1) circle (0.5ex);
\fill (-5,0) circle (0.5ex);
\fill (-6,-1) circle (0.5ex);
\fill (-6,-2) circle (0.5ex);
\fill (-5,-1) circle (0.5ex);
\draw (-4.5,1) -- (-6,0) -- (-6,-1) -- (-6,-2);
\draw (-4.5,1) -- (-5,0) -- (-5,-1);
\draw (-4.5,1) -- (-4,0) -- (-4,-1);
\draw (-4.5,1) -- (-3,0);

\fill (-5.5,1.5) circle (0.5ex);
\fill (-3.5,1.5) circle (0.5ex);
\fill (-5.5,2.5) circle (0.5ex);
\fill (-3.5,2.5) circle (0.5ex);
\draw (-4.5,3) -- (-5.5,2.5)  --(-5.5,1.5) -- (-4.5,1) -- (-3.5,1.5) -- (-3.5,2.5) -- (-4.5,3);
\draw (-5.5,2.5) -- (-3.5,2.5) -- (-5.5,1.5) --(-3.5,1.5) --(-5.5,2.5) -- (-4.5,1) -- (-3.5,2.5) ;
\fill (-4.5,3) circle (0.5ex);
\draw (-4.5,1) -- (-4.5,3);
\draw (-3.5,1.5) -- (-4.5,3) -- (-5.5,1.5);
\node at (-4.5,0.5) {$u$};
\node at (-6,1.5) {$v_1$};
\node at (-6,2.5) {$v_2$};
\node at (-4.5,3.5) {$v_3$};
\node at (-3,2.5) {$v_4$};
\node at (-3,1.5) {$v_5$};
\end{tikzpicture}\caption{Pineapple graph $Pi(6,(3,2,2,1))$}\label{fig-pineapple}
\end{figure}

\begin{thm}\label{thm-pineapple-lc}
For all $n \ge 1$ and any partition $\lambda$, the independence polynomial of $Pi(n,\lambda)$ is log-concave.
\end{thm}
\begin{proof}
 In fact $Pi(n,\lambda) = S((\lambda,1))_v^{n-1}$, where $S((\lambda,1))$ is the graph obtained by attaching an edge $uv$ to the torso $u$ in $S(\lambda)$. Taking $G=S((\lambda,1))$ in Proposition \ref{prop-rec-clan} leads to
\[
Y_{Pi(n,\lambda)} = Y_{S(\lambda)} + (Y_{K_{n-1}}-1)Y_{S(\lambda)-u} + \sum_{\alpha}X_{Pi(n,\lambda)}^{\alpha},
\]
where $\alpha(u) \ge 1$ and $\alpha(v_i) \ge 1$ for some $1 \le i \le n-1$.
Then we analyze these $\alpha$'s more precisely.
\begin{itemize}
  \item[(1)] $\alpha(u) \ge 1$, $\alpha(v_i)\ge 1$ and $\alpha(v_j)\ge 1$ for some $1 \le i < j \le n-1$: $X_{{Pi(n,\lambda)}^{\alpha}} =_{2s} 0$.
  \item[(2)] $\alpha(u) \ge 1$, $\alpha(v_i) \ge 1$ for some $1 \le i \le n-1$ and $\alpha(v_j) = 0$ for any $1 \le j \le n-1$ and $j \neq i$:
      \begin{itemize}
        \item[(2.1)] $\alpha(u) \ge 2 $ or $\alpha(v_i) \ge 2$. In this subcase $X_{{Pi(n,\lambda)}^{\alpha}} =_{2s} 0$.
        \item[(2.2)] $\alpha(u) = \alpha(v_i) = 1$. In this subcase $X_{{Pi(n,\lambda)}^{\alpha}}$ is (possibly) not 2-$s$-positive.
      \end{itemize}
\end{itemize}

Secondly, we shall define an injection from \textbf{(2.2)} to the $\alpha$'s in $(Y_{K_{n-1}}-1)Y_{S(\lambda)-u}$, where we directly let $\phi(\alpha)(u) = 0$ and $\phi(\alpha)(v_i) = 2$. Actually, this is case (i) of the map $\phi$ (defined in proof of Theorem \ref{thm-spi-lc}) restricted to the spider $S((\lambda,1))$, where we regard $uv_i$ as a leg of $u$. Then by the proof of Theorem \ref{thm-spi-lc} all the negative terms in $\sum_{\alpha}X_{Pi(n,\lambda)}^{\alpha}$ will be eliminated. In addition, the terms not being images of $\phi$ are also 2-$s$-positive since $S(\lambda)-u$ consists of only paths.  Then the proof follows since $Y_{S(\lambda)}$ is 2-$s$-positive. %
\end{proof}

\vspace{0.5cm}

\noindent{\bf Acknowledgments.}
Ethan Li is supported by the Fundamental Research Funds for the Central Universities (GK202207023).
Arthur Yang is supported by the National Science Foundation of China (12325111).

\printbibliography

\end{document}